\documentclass[11pt]{amsart}
\usepackage[latin1]{inputenc}
\usepackage{mathrsfs}
\usepackage{mathtools}
\usepackage{amsmath, amsthm, amsfonts, amssymb}
\usepackage{enumitem}
\usepackage{tikz}
\usepackage{tikz-cd}
\usepackage[all]{xy}
\usepackage{xcolor}

\usepackage[left=2.3cm,top=3.3cm,right=2.3cm,bottom=2.7cm]{geometry}

\theoremstyle{plain}
\newtheorem{thm}{Theorem}[section]

\newtheorem{cor}[thm]{Corollary}

\newtheorem{prop}[thm]{Proposition}
\newtheorem{problem}[thm]{Problem}
\theoremstyle{definition}
\newtheorem{example}[thm]{Example}

\newtheorem{rmk}[thm]{Remark}

\numberwithin{equation}{section}

 \usepackage{todonotes}

\newcommand{\alb}{\mathrm{alb}}

\newcommand{\ev}{\mathrm{ev}}

\newcommand{\im}{\mathrm{Im}}

\newcommand{\supp}{\mathrm{supp}}

\newcommand{\CHM}{\mathrm{\CHM}}

\usepackage[
urlcolor=blue,
colorlinks=true,
linkcolor=blue,
citecolor=blue,
]{hyperref}

\newcommand{\CC}{\mathbb{C}}
\newcommand{\QQ}{\mathbb{Q}}

\newcommand{\PP}{\mathbb{P}}

\newcommand{\sO}{\mathcal{O}}

\begin{document}
    \title[On the canonical degree of a threefold of general type]{On the canonical degree of a Gorenstein minimal threefold of general type}
	\author{Jiabin Du}
    \address{Shanghai Institute for Mathematics and Interdisciplinary Sciences (SIMIS)\\Shanghai
200433, China}
\address{Research Institute of Intelligent Complex Systems\\Fudan University\\Shanghai 200433, China}
    \email{jiabin.du@simis.cn}
    
    \author{Yong Hu}
    \address{ School of Mathematical Sciences \\Shanghai Jiao Tong University\\ Dongchuan  Road 800  \\ Shanghai 200240, China}
    \email{yonghu@sjtu.edu.cn}
	
	\date{\today}
		\begin{abstract}
		Let $X$ be a Gorenstein minimal   $3$-fold of general type whose canonical map is generically finite. We prove that if  $p_g(X)> 243$, then the  degree of the canonical map is at most $72$. Moreover, equality holds  only if the general fibre $F$ of the Albanese morphism of $X$ is a smooth  minimal surface of general type satisfying $p_g(F)=3,q(F)=0$ and $K_F^2=36$, and the canonical map of $F$ has degree $36$. This result improves the lower bound on  $p_g(X)$ previously obtained by Jin-Xing Cai~\cite{Cai08}.  

        As a consequence, we show that if the canonical degree is bigger than $64$, then the general fibre of the Albanese morphism of $X$ is a surface with irregularity zero.
	\end{abstract}
\maketitle

\section{Introduction}

The canonical map is a fundamental tool in the study of the explicit birational geometry of varieties of general type. For a smooth  curve $C$ of genus $g\geq 2$, it's well-known that the canonical map of $C$ is either an  embedding or a finite morphism of degree $2$.  Let $S$ be a smooth  surface of general type whose canonical map is  generically finite.  Beauville \cite{Bea79} proved that the degree of the canonical map of $S$ is bounded.  Xiao \cite{Xia86} showed that if $p_g(S)>132$, then the canonical map of $S$ has degree at most $8$.

For a Gorenstein minimal $3$-fold of general type with generically finite canonical map,  Meng Chen \cite[\S 2.10]{Chen05} posted the following problem :
\begin{problem}
Let $X$ be a Gorenstein minimal 3-fold with at worst locally factorial terminal singularities. Suppose the canonical map $\phi_{X}$ of $X$ is generically finite onto its image. Is the generic degree $d_X$ of $\phi_X$ universally upper bounded ?
\end{problem}
An affirmative answer was given by Hacon \cite{Hac04}, who  established the bound that $d_X\leq 576$. Subsequently, Du and Gao \cite{DG16} improved this to  $d_X\leq 360$. Later Du \cite{Du18} obtained a further refinement under an additional assumption on the canonical image of $\phi_X$.

It is natural to expect that, as in the surface case, such an upper bound should become smaller when the geometric genus is large. In this direction,  Cai \cite{Cai08} proved that $d_X\leq 72$ whenever $p_g(X)>105411$.  

Our first main result is the following
\begin{thm}\label{thm:main}
   Let $X$ be a Gorenstein minimal  threefold of general type with canonical divisor $K_X$. Suppose that $|K_X|$ defines a generically finite map $\phi_X$ of degree $d$. If $p_g(X)>243$, then $d\leq72$. 
   
   Moreover, if $d=72$, then  the general Albanese fibre $F$ of  $X$ is a smooth  minimal surface of general type satisfying $p_g(F)=3,q(F)=0$ and $K_F^2=36$, and the canonical map of $F$ has degree $36$.
\end{thm}

\begin{rmk}
\begin{itemize}
    \item[(1)]  When $p_g(X)$ is small, examples of smooth minimal threefolds of general type with canonical map of degree $96$ do exist (see \cite{FG20}).
    \item[(2)] If $X$ is not assumed to be Gorenstein minimal, Hacon \cite{Hac04} proved that the canonical degree $d$ is unbounded.
\end{itemize} 
\end{rmk}
 Under the same assumption as in Theorem \ref{thm:main}, we have
\begin{cor}\label{cor:main}
    If  $d>64$, then  the general Albanese fibre of $X$ is a regular surface.
\end{cor}

\subsection{ Notation and Conventions.}

Throughout the paper, we work over the complex number field $\CC$. We follow the standard definitions and notation in \cite{KM98}.
All varieties are assumed to be projective. 

\subsubsection*{Varities and divisors}
For a normal variety $X$, we denote by $K_X$ for the canonical divisor.  We write $D\ge 0$ to indicate that $D$ is an effective divisor, and denote by $|M|$ the linear system associated to a divisor $M$ on $X$.
For two $\mathbb{Q}$-divisors $D_1, D_2$, we write $D_1 \ge D_2$ if $D_1-D_2$ is effective. 
\subsubsection*{Irregular varieties} Let $X$ be a normal variety with at worst rational singularities. We say that $X$ is \emph{irregular}, if $q(X) := h^1(X, \mathcal{O}_X) > 0$. Note that $X$ has a well-defined \emph{Albanese map}
$$
a: X \to \mathrm{Alb}(X),
$$
where $A:=\mathrm{Alb}(X)$ is an abelian variety referred to as the Albanese variety of $X$. The number $\dim a(X)$, denoted by $\alb.\dim X$, is called the \emph{Albanese dimension} of $X$. Let 
$$
X \stackrel{f}\rightarrow Y \stackrel{h}\rightarrow A
$$ 
be the Stein factorization of $a$. Then $f$ is called the \emph{Albanese fibration} of $X$, and a fibre of $f$ is called an \emph{Albanese fibre}.



\subsection*{Acknowledgments}
The authors thank  Shanghai Institute for Mathematics and Interdisciplinary Sciences (SIMIS) for providing a stimulating research environment for discussions of this work. 
Yong Hu is supported by National Key Research and Development Program of China \#2023YFA1010600 and the National Natural Science Foundation of China (Grant No. 12571044).

\section{Preliminaries}

\subsection{Image sheaf of the evaluation map}
Let $C$ be a smooth  curve of genus $g$, and let $E$ be a vector bundle of rank $r$ on $C$. Consider the evaluation map
\[
\ev\colon H^0(E)\otimes\sO_C\to E
\]
Denote by $E_{\ev}$ the image sheaf of $\ev$, which is a vector bundle, say of rank $l$.
\begin{thm}\cite[Proposition~page 476]{Xia86}\label{Xiao_Eev}
    Suppose that $E\otimes\omega_C$ is nef, then
    \[
    (g-1)(r-l)\leq l.
    \]
    Moreover, when the equality holds, $E$ is a semi-stable bundle of degree $(2g-2)r$.
\end{thm}

\subsection{Miyaoka-Yau inequality}

For minimal varieties of general type, Greb, Kebekus, Peternell and Taji (see \cite[Theorem 1.1]{GKPT19} and \cite[line 12 in \S 1.1, page 1488]{GKPT19}) proved the following $\QQ$-Miyaoka-Yau inequality:

\begin{thm}[$\mathbb{Q}$-Miyaoka-Yau inequality]\label{thm: MY1}
Let $X$ be an $n$-dimensional, minimal variety of general type. Then,
\begin{equation*}
\bigl(2(n+1)\cdot c_2(X) - n\cdot c_1(X)^2\bigr)\cdot [K_X]^{n-2} \geq 0.
\end{equation*}
\end{thm}
In particular, we have
\begin{cor}
    Let $X$ be a Gorenstein minimal $3$-fold of general type, then 
    \begin{equation}\label{MY}
        K_X^3\leq 64\chi(\omega_X).
    \end{equation}
   
\end{cor}
\begin{proof}
    By Theorem \ref{thm: MY1}, we have 
    $$
    8(K_X\cdot c_2(X))\geq 3K_X^3.
    $$
    By \cite[Corollary~10.3]{Rei87}, we have $(K_X\cdot c_2(X))=24\chi(\omega_X)$. It follows easily that 
    $$
    K_X^3\leq 64\chi(\omega_X).
    $$
    The proof is completed.
\end{proof}

\section{Proofs of~Theorem \ref{thm:main} and Corollary~\ref{cor:main}} 
Let $X$ be a Gorenstein minimal  threefold of general type. Suppose that $|K_X|$ defines a generically finite map $\phi_X\colon X\dashrightarrow \Sigma\subset\PP^{p_g(X)-1}$, where $\Sigma=\im(\phi_X)$.  Let $d$ be the generic degree of $\phi_X$,
by the Miyaoka-Yau inequality (\ref{MY}), we have
\[
d\deg(\Sigma)\leq K_X^3\leq 64\chi(\omega_X).
\]
Since $\Sigma$ is nondegenerate, $\deg(\Sigma)\geq p_g(X)-3$. It follows that
\begin{equation}\label{MY-bound}
    d\leq \frac{K_X^3}{\deg(\Sigma)}\leq \frac{K_X^3}{p_g(X)-3}\leq \frac{64\chi(\omega_X)}{p_g(X)-3}.
\end{equation}

The proofs of Theorem~\ref{thm:main} and Corollary~\ref{cor:main} follow directly from  two propositions below.

\begin{prop}\label{keyprop1}
Suppose that $p_g(X)>243$. Then we have $d\leq 64$, unless the general Albanese fibre of $X$ is a regular surface.
\end{prop}
\begin{proof}
By \cite[Proposition 2.1]{CH06} and our assumption, we have $\chi(\omega_X)\leq p_g(X)$.
Since $p_g(X)>243$, by \eqref{MY-bound}, we have
\[
d\leq \frac{64\chi(\omega_X)}{p_g(X)-3}\leq \frac{64p_g(X)}{p_g(X)-3}< 65.
\]
The proof is completed.

\end{proof}
By adopting the idea from the proof of \cite[Theorem 5]{Xia86}, we obtain the following proposition.
\begin{prop}\label{keyprop2}
    Suppose that $p_g(X)>243$ and the general Albanese fibre $F$ of $X$ is regular. Then we have $d\leq 72$. In the case $d=72$, the general fibre $F$ is a smooth  minimal surface of general type with $p_g(F)=3,q(F)=0$ and $K_F^2=36$ whose canonical degree is $36$.
\end{prop}
\begin{proof}
Note that the Albanese fibration $f: X\to Y$ is a fibration onto a smooth curve $Y$ of genus $b=q(X)$. By \cite[Proposition 2.1]{CH06}, we have
\begin{align}\label{ineq:pg}
\chi(\omega_X)\leq p_g(X)(1+\frac{1}{p_g(F)}).
\end{align}
Remark that we have $$
p_g(X)=h^0(f_{\ast}\omega_X)\geq \chi(f_{\ast}\omega_X)\geq p_g(F)(q(X)-1).
$$
We proceed with the proof by three cases:

{\bf Case 1.}  $p_g(F)\geq 9$. 

We deduce that
\[
d\leq \frac{64\chi(\omega_X)}{p_g(X)-3}\leq \frac{64p_g(X)(1+\frac{1}{p_g(F)})}{p_g(X)-3}< 72.
\]

{\bf Case 2.} $p_g(F)\leq 8$ and $b=q(X)\leq 7$.

In this case, we have $\chi(\omega_X)\leq p_g(X)+6$. We deduce that
$$
d\leq \frac{64\chi(\omega_X)}{p_g(X)-3}\leq \frac{64(p_g(X)+6)}{p_g(X)-3}<67.
$$

{\bf Case 3.}  $p_g(F)\leq 8$ and $b=q(X)\geq 8$. 

Up to  some birational modifications of $X$ and $\Sigma$, we may assume that $\phi_X$ is a morphism, $\Sigma$ is smooth. Let $H$ be a very ample divisor on $Y$. Choose two elements $D_1,D_2\in |H|$ with $\supp D_1\cap\supp D_2=\emptyset$. Put $\widetilde{D_i}=\phi_{X,\ast}f^{\ast}D_i$ for $i=1,2$. Then the linear pencil spanned by $\widetilde{D_1}$ and $\widetilde{D_2}$ induces a rational map $\Sigma\dashrightarrow \PP^1$. After taking some higher birational model, we may assume that the rational map $\Sigma\dashrightarrow \PP^1$ is  a morphism. Let $h\colon \Sigma\to C$ be a fibration obtained from the Stein factorization of $\Sigma\to \PP^1$, then we have the following commutative diagram
\[
\begin{tikzcd}
 X \arrow[r, "\phi_X"]\arrow[d, "f"] & \Sigma \arrow[d,"h"]\\
Y \arrow[r,"\delta"] & C
\end{tikzcd}
\]
The image of the evaluation map
\[
\begin{tikzcd}
 \ev\colon H^0(f_{\ast}\omega_X)\otimes\sO_Y \arrow[r] & f_{\ast}\omega_X
\end{tikzcd}
\]
generates a rank $r$ subsheaf $(f_{\ast}\omega_X)_{\ev}$ of $f_{\ast}\omega_X$. We consider separately the two subcases.
\subsubsection*{Subcase $r<p_g(F)$} 
 By Theorem~\ref{Xiao_Eev}, we deduce that 
 $$
 b-1\leq(b-1)(p_g(F)-r)\leq r\leq p_g(F)-1.
 $$
 It follows that $b\leq p_g(F)\leq 8$. Note that 
 $$
\chi(\omega_X)=\chi(f_{\ast}\omega_X)+b-1\leq p_g(X)+b-1\leq p_g(X)+7.
$$
Then we have
\[
d\leq \frac{64\chi(\omega_X)}{p_g(X)-3}\leq \frac{64(p_g(X)+7)}{p_g(X)-3}< 67.
\]
\subsubsection*{Subcase $r=p_g(F)$} In this case, the natural restriction map $$H^0(K_X)\to H^0(K_F)$$
is surjective. We conclude that
\begin{equation}\label{equation:cdeg}
    d=\deg(\phi_F)\cdot\deg(\delta).
\end{equation}
By adopting the idea in \cite{Bea79}, one can prove that $\deg(\phi_F)\leq 36$ (see \cite[Proposition 1.1]{LY21} for the proof).
We are left to bound $\deg(\delta)$. To this end, take a global section $t\in H^0(f_{\ast}\omega_X)$ such that the zero divisor $D_t$  has maximal degree.

If $\deg (D_t)>\frac{3}{2}b$, then by Riemann-Roch theorem, we deduce that
\[
h^0(D_t)\geq \deg (D_t)+1-b>\frac{1}{3}\deg (D_t)+1.
\]
Consider the rational map induced by linear system $|D_t|$:
\[
\begin{tikzcd}
\phi_{|D_t|} \colon  Y \arrow[r] & \Sigma_t\subset\PP^{h^0(D_t)-1}
\end{tikzcd}
\]
then $\deg (D_t)\geq \deg\phi_{|D_t|}\cdot\deg(\Sigma_t)>\frac{1}{3}\deg (D_t)\cdot \deg(\phi_{|D_t|})$. So we have $ \deg(\phi_{|D_t|})\leq 2$.  Since $f^{\ast}|D_t|$ is a sub-linear system of $|K_X|$, the map $\phi_{|D_t|}$ factors through $\delta$. It follows that $\deg(\delta)\leq\deg (\phi_{|D_t|})\leq 2$.

Now we consider the case where $\deg (D_t)\leq \frac{3}{2}b$. Let $D$ be a sum of $b-2$ general points on $Y$. By Riemann-Roch theorem, we deduce that 
\begin{align*}
h^0(f_{\ast}\omega_X\otimes\sO_Y(-D))&\geq \chi(f_{\ast}\omega_X\otimes\sO_Y(-D))\\
&\geq \deg( f_{\ast}\omega_X\otimes\sO_Y(-D))-p_g(F)(b-1)\\
    &\geq p_g(F)>0.
\end{align*}
Let $s\in H^0(Y,f_{\ast}\omega_X\otimes\sO_Y(-D))$ be a nonzero section, then after twisting by $\sO_Y(D)$ yields a section $s_1$ of $f_{\ast}\omega_X$, with the corresponding divisor $D_1$ satisfying $D_1=D+div(s)\geq D$. Hence $\deg (\delta_{\ast}D_1)\geq b-2$. The rational map $\phi_{|D_1|}$, defined by linear system $|D_1|$, factors through $\delta$, and $t\in H^0(f_{\ast}\omega_X)$ has maximal degree, we infer that
\[
\frac{3}{2}b\geq \deg (D_t)\geq \deg (D_1)\geq\deg(\delta)\cdot(b-2),
\]
which implies that $\deg(\delta)\leq 2$ since $b=q(X)\geq 8$. 

 Thus we always have $\deg(\delta)\leq 2$. We conclude that $\deg(\phi_X)\leq 72$ by equality~(\ref{equation:cdeg}).

For the case $d=72$. By the above discussion, we know that $\deg (\phi_F)=36$ and $\deg(\delta)=2$. It follows that $F$ is a smooth minimal   surface of general type with $p_g(F)=3, q(F)=0$ and $K_F^2=36$. The proof is completed.
\end{proof}
An explicit example of a threefold of general type with canonical degree 
$72$ is easy to construct, building on the construction of surfaces of general type with maximal canonical degree $36$ (see \cite{LY21}  and \cite{Rit22}).
\begin{example}
    Let $S$ be a minimal surface of general type with $p_g(S)=3, q(S)=0, K_S^2=36$ and the canonical map of $S$ is generically finite of degree $36$. Let $C$ be a hyperelliptic curve of genus $b\geq 2$. Consider the product $3$-fold $X:=S\times C$, then $X$ has numerical invariants :
    \[
    p_g(X)=3b, q(X)=b, \deg(\phi_X)=72.
    \]
\end{example}

\bibliographystyle{alpha}
\bibliography{ref}

\end{document}